\documentclass[11pt]{article}
\usepackage{amsmath,amssymb,amsthm,amscd,dsfont}
\usepackage{amsfonts,latexsym,rawfonts,amsmath,amssymb,amsthm}
\usepackage{lscape}
\usepackage{amscd, float,times,rotating}
\usepackage{pb-diagram}
\usepackage{hyperref}
\usepackage[pagewise]{lineno}

\numberwithin{equation}{section}

\newtheorem{prop}{Proposition}[section]
\newtheorem{theorem}[prop]{Theorem}
\newtheorem{lemma}[prop]{Lemma}
\newtheorem{corollary}[prop]{Corollary}

\newtheorem{example}[prop]{Example}
\newtheorem{definition}[prop]{Definition}

\def\begeq{\begin{equation}}
\def\endeq{\end{equation}}

\begin{document}
\title{The relative volume function and the capacity of sphere on asymptotically hyperbolic manifolds}
\author{Xiaoshang Jin}
\date{}
\maketitle
\begin{abstract}
Following the work of Li-Shi-Qing, we propose the definition of the relative volume function for an AH manifold. It is not a constant function in general and we study the regularity of this function. We use this function to give an accurate characterization of the height of the geodesic defining function for the AH manifold with a given boundary metric. It is also proved that such functions
are uniformly bounded from below at infinity and the bound only depends on the dimension. As an application, we use this function to research  the capacity of balls in AH manifold and provide some limit results.
\end{abstract}
\section{Introduction}  \label{sect1}
Let $\overline{X}$ be an $n+1$ dimensional smooth manifold with boundary. We use $X$ and $\partial X$ to denote the interior and the boundary of $\overline{X}.$ A complete noncompact metric $g^+$ on $X$ is called conformally compact if there exists a smooth defining function $\rho$ on  $\overline{X}$ such that the conformal metric $g=\rho^2g^+$ can be continuously extended to $\overline{X}.$
The  defining function satisfies that
\begin{equation}\label{1.1}
  \rho>0\ \ in \ \ X,\ \ \ \ \rho=0\ \ on\ \ \partial X,\ \ \ \ d\rho\neq 0\ \ on \ \ \partial X.
\end{equation}
We say that $(X,g^+)$ is $C^{m,\alpha}\ (W^{k,p},$ smoothly) conformally compact if $g=\rho^2g^+$ is $C^{m,\alpha}\ (W^{k,p},$ smooth) on $\overline{X}.$ Here $m,k\in \mathbb{N},\alpha\in[0,1)$ and $p\geq 1.$  The restriction of $g$ on the boundary $\hat{g}=g|_{\partial X}$ is called the boundary metric. It is well known that $g^+$ induces a conformal class $(\partial X,[\hat{g}])$ on the boundary, called the conformal infinity of $g^+.$ If $g^+$ is $C^2$ conformally compact, then a straightforward calculation yields that
the curvature of $(X,g^+)$ is of the following from \cite{graham2000volume}:
\begin{equation}\label{1.2}
  R_{ijkl}[g^+]=|d\rho|^2_{g}(g^+_{ik}g^+_{jl}-g^+_{il}g^+_{jk})+O(\rho^{-3})
\end{equation}
near $\partial X.$ Here $O(\rho^{-3})$ denotes the (0,4) tensors and takes norm with respect to metric $g.$  Readers can see \cite{besse2007einstein} for the conformal transformation
law of curvatures. Hence the sectional curvature $K[g^+]=-|d\rho|^2_{g}+O(\rho)$ is uniformly approaching to $-|d\rho|^2_{g}$ (see \cite{mazzeo1988hodge}).
\par Thus if $g^+$ is at least $C^2$ conformally compact and $|d\rho|^2_{g}|_{\partial X}=1,$ we say $(X,g^+)$ is an asymptotically hyperbolic manifold or AH manifold for short. Given an  AH manifold $(X,g^+)$ and a boundary representative $\hat{g}\in[\hat{g}],$ Graham and Lee
proved that there exists a unique defining function $x$ such that $|dx|^2_{x^2g^+}=1$ in a neighborhood of the boundary $\partial X\times[0,\delta)$
for some small $\delta$ in \cite{graham1991einstein}. We call $x$ the geodesic defining function of $(X,g^+)$ associated with $\hat{g}.$
However, the new geodesic compactification may not be as smooth as the initial one. In fact, if $g=\rho^2g^+$ is a $C^{m,\alpha}$ compactification with $m \geq 2$ and $\alpha\geq 0,$ then the geodesic compactification $\bar{g}=x^2g^+$ with the same boundary metric is, at least $C^{m-1,\alpha},$ see \cite{lee1994spectrum} or \cite{chrusciel2005boundary}.
\par Let $\delta>0$ be the supremum where the geodesic defining function is well defined as above. One interesting problem is: how big (small) is $\delta?$ It seems that $\delta$ is dependent on $g^+$ and $\hat{g},$  can we give a mathematical expression between them? We answer the problems in this paper.
\\ \par If in addition, the AH manifold $(X,g^+)$ also satisfies the Einstein equation,
$$ Ric[g^+]=-ng^+$$
 then we call $g^+$ a conformally compact Einstein metric or Poincar\'e-Einstein metric.  In recent years it has become the main theme in the study of conformal geometry as it plays an important role in the proposal of AdS/CFT correspondence \cite{maldacena1999large} in theoretic physics.
\par In this paper, we mainly research the AH manifold of order $2,$ i.e. the sectional curvature
satisfies that
\begin{equation}\label{1.3}
  K[g^+]+1=O(\rho^2)
\end{equation}
We emphasize that a conformally compact Einstein metric $g^+$ satisfies the curvature condition (\ref{1.3}), (one can see section 2 in \cite{jin2021finite} or lemma 1.4 in \cite{li2017gap}, or lemma 3.1 in \cite{chrusciel2005boundary} for more information). Therefore we also call $g^+$
an asymptotically hyperbolic Einstein (or AHE for short) metric.
\\ \par The motivation for this paper is mainly from the work of Li-Shi-Qing when they study the gap phenomena and curvature estimates for AH manifolds. They proved that:
\begin{lemma}\cite{li2017gap}
  Let $(X^{n+1},g^+)$ be an $n+1-$ dimensional AH manifold with a $C^2$ conformally compactification $g=\rho^2g^+,$ if the sectional curvature satisfies (\ref{1.3}), then for any $p\in X,\frac{Vol(\partial B_{g^+}(p,t))}{\sinh^nt}$ is convergent as $t$ tends to infinity.
\end{lemma}
Here we illustrate that the original thought is from Shi and Tian (lemma 2.2 in \cite{shi2005rigidity}) where they studied the rigidity of
the ALH manifold with an intrinsic geometric condition. By analyzing the geodesic sphere of radius $\rho$ and centered at a fixed pole $o,$ they
showed that $e^{-2\rho}g^+$ is uniformly equivalent to the flat metric $\delta_{ij}$ and hence $\frac{Vol(\partial
B_{g^+}(o,\rho))}{\sinh^n\rho}$ is uniformly bounded. Furthermore, the curvature  and diameter of $(\partial
B_{g^+}(o,\rho),\sinh^{-2}\rho g^+)$ are also bounded. Then they proved that $(\partial B_{g^+}(o,\rho),\sinh^{-2}\rho g^+)$ converges in
the weakly $W^{2,p}$-topology to a manifold which is conformally equivalent to a sphere and finally gets a rigidity result. Later the research was developed by Dutta and Javaheri in \cite{dutta2010rigidity} and they assumed that the conformal infinity of the AH manifold is
the standard round sphere metric. They compared the distance functions $t$ and $r=-\ln\frac{x}{2}$ of $g_+$ and showed that $t-r$ has a
continuous extension to the boundary $\mathbb{S}^n$ and $(\Sigma_r=\{y:r(y)=r\},4e^{-2t}g^+)$ converges in the weak $W^{1,2}$ topology to some metric on the sphere. Finally, by computing the volume of the hypersurface they obtained the rigidity result. A more general case was studied by Li-Shi-Qing in \cite{li2017gap} and they get a relative volume inequality for AHE manifolds. This inequality also provides a new curvature pinching estimate for AHE manifolds with conformal infinities having large Yamabe invariants.
\par It is natural to ask: whether the limit in lemma 1.1 is dependent on the choice of the base point $p?$ As shown in example \ref{ex4.2}, the answer is yes, even for hyperbolic manifold. Then it is meaningful to introduce the following definition:
\begin{definition}
Suppose that $(X,g^+)$ is an $n+1-$ dimensional AH manifold of order 2, then for any $p\in X,$ we define
\begin{equation}\label{1.4}
  \mathcal{A}(p)=2^n\cdot\lim\limits_{t\rightarrow+\infty}e^{-nt}\cdot Vol(\partial B_{g^+}(p,t))
\end{equation}
we call $\mathcal{A}$  the relative volume function on $X.$
\end{definition}
We remark that condition (\ref{1.3})  is a sufficient and unnecessary condition for the existence of the limit. In fact, if $Ric[g^+]\geq ng^+,$ then we can still define the relative volume function because of the volume comparison theorem.
\\ \par In the first part of this paper, we will study the property of $\mathcal{A}.$ Firstly, we use the triangle inequality to prove that $\ln\mathcal{A}$ is a Lipschitz continuous function. Then we show that there is a relationship between $\mathcal{A}$ and the 'height' of the geodesic defining function.
More concretely,
\begin{theorem}
  Let $(X^{n+1},g^+)$ be an $n+1-$ dimensional AH manifold of order 2, then we have the following:\\
  (1) The function$\mathcal{A}$ is positive on $X$ and $\ln\mathcal{A}$ is Lipschitz continuous on $X$ in the sense that
\begin{equation}\label{1.5}
  |d \ln\mathcal{A}|_{g^+}\leq n\ \ i.e.\ \ |d \mathcal{A}|_{g^+}\leq n\mathcal{A}.
\end{equation}
  (2) Let $g=x^2g^+$ be the $C^2$ geodesic compactification of $g^+$ and the boundary metric of $g$ is $\hat{g}=g|_{\partial X}.$ If $|dx|^2_g=1$ and $x$ is smooth on $X_\delta=\partial X\times(0,\delta),$ then
  for any $p\in X\setminus X_\delta,$
\begin{equation}\label{1.6}
  \mathcal{A}(p)\leq(\frac{2}{\delta})^n\cdot Vol(\partial X,\hat{g})\leq e^{n\cdot diam(X\setminus X_\delta,g^+)}\mathcal{A}(p).
\end{equation}
 Here $diam(X\setminus X_\delta,g^+)$ denotes the diameter of $X\setminus X_\delta$ in $(X,g^+).$
\end{theorem}
We state that (\ref{1.6}) provides the above and lower bounds of $\delta$ which is the "height "of geodesic defining function. On one hand, if $Ric[g^+]\geq -ng^+,$ then we have the following estimates of lower bound in some sense.
\begin{equation}\label{1.7}
 \delta\cdot e^{diam(X\setminus X_\delta, g^+)}\geq 2\cdot(\frac{Vol(\partial X,\hat{g})}{\mathcal{A}(p)})^{\frac{1}{n}}\geq 2\cdot(\frac{Vol(\partial X,\hat{g})}{Vol(\mathbb{S}^n,g_{\mathbb{S}})})^{\frac{1}{n}}
\end{equation}
If the complement set of $X_\delta$  has a
uniformly bounded diameter, then the height of the geodesic function $\delta$ is bounded from below. This is very useful when we study the
compactness problems of AHE metrics.
\par On the other hand, let's first recall the classic relative volume inequality in \cite{li2017gap}.  The authors showed that
the  relative volume is uniformly bounded from below by the Yamabe constant of the conformal infinity. As an application of the relative volume inequality and (\ref{1.6}), we get the following estimates of upper bound.
\begin{corollary}
  Let $(X^{n+1}, g^+)$ be an AH manifold of $C^3$ regularity whose
conformal infinity $(\partial X,[\hat{g}])$ is of positive Yamabe type. Let $p\in X^{n+1}$ be a fixed point. Assume
further that
\begin{equation} \label{1.8}
   Ric[g^+] \geq -ng^+ \ \ and \ \ R[g^+] + n(n +1) = o(e^{-2t})
\end{equation}
for the distance function t from p. Let $g=x^2g^+$ be a $C^2$ geodesic compactification with boundary metric $\hat{g}$ and $\delta$  be the supremum of $x$ where $|dx|_g^2=1$ and $x$ is smooth on $\partial X\times(0,\delta).$
 Then we have the upper bound estimate of $\delta:$
\begin{equation} \label{1.9}
  \delta\leq 2\cdot (\frac{Vol(\partial X,\hat{g})}{Vol(\mathbb{S}^n,g_{\mathbb{S}})})^{\frac{1}{n}}\cdot(\frac{Y(\mathbb{S}^n,[g_{\mathbb{S}}])}{Y(\partial X,[\hat{g}])})^{\frac{1}{2}}
\end{equation}
\end{corollary}
The inequalities (\ref{1.7}) and (\ref{1.9}) are sharp. If we consider the hyperbolic space: (Poincar\'e ball model)
$$(\mathbb{B}^{n+1},\ g_\mathbb{H}=(\frac{2}{1-|x|^2})^2|dx|^2),$$
where $\rho=\frac{1-|x|^2}{2}$ is defining function.  Let $s=2\frac{1-|x|}{1+|x|},$ then
$$g_\mathbb{H}=s^{-2}(ds^2+\frac{(4-s^2)^2}{16}g_{\mathbb{S}^n}),$$
where $s$ is the geodesic defining function. Hence the conformal infinity is the standard conformal round sphere $(\mathbb{S}^n,[g_{\mathbb{S}^n}]).$
We find that $\delta=2$ and the equality in (\ref{1.7}) and (\ref{1.9}) holds.
\par We also consider the asymptotical behavior of the relative volume function and show that $\mathcal{A}(p_0)$ is uniformly bounded from below by a constant only depending on the dimension $n+1$ when $p_0$ approaches to infinity.
\begin{theorem}
   Let $(X^{n+1},g^+)$ be an $n+1-$ dimensional AH manifold of order 2, then
  \begin{equation}\label{1.10}
    \liminf\limits_{p_0\rightarrow \partial X}\mathcal{A}(p_0)\geq (\frac{2}{n})^n(n-1)!\omega_{n-1}.
  \end{equation}
  As an application, the function $\mathcal{A}$ is uniformly bounded from below by a positive number.
\end{theorem}

 \par In the second part of this paper. We do some calculations on some AH manifolds. Firstly, we show that the relative volume function on hyperbolic space is exactly the volume of the standard sphere $\omega_n$ and the inverse conclusion also holds provided the Ricci curvature is bounded from below. More concretely, we have the following rigidity result:
\begin{theorem}
  Let $(X,g^+)$ be an $n+1-$ dimensional AH metric and the Ricci curvature satisfies $Ric[g^+]\geq -ng^+$, then $\mathcal{A}\leq \omega_n$ and the equality holds at one point if and only if $X$ is isometric to a hyperbolic space.
\end{theorem}
We also calculate the relative volume function on some hyperbolic manifolds and obtain that the relative volume functions are not constants on these spaces. See example \ref{ex4.2}.
\\
\par In the third part of paper, we provide an application of the relative volume function. Recall the work in \cite{bray2008capacity}, where the authors derive an upper bound for the capacity of an asymptotically flat 3-manifold with nonnegative scalar curvature and also showed a rigidity result, i.e. the upper bound is achieved if and only if $(M^3,g)$ is isometric to a spatial Schwarzschild manifold.
\par In this paper, we study the capacity of an AH manifold. Since the pivotal of an AH manifold is that it has a good structure at infinity, i.e. a compactification, then we will study the behavior of the capacity at infinity. Using the relative volume function, we give an upper bound of the capacity of balls in AH manifold when its radius goes to infinity. More concretely, it is easy to calculate the capacity of balls in hyperbolic space, it has the formula (\cite{grigor1999isoperimetric}):
\begin{equation}\label{1.11}
 \lim\limits_{t\rightarrow+\infty}\frac{cap_p(B_\mathbb{H}(o,t))}{e^{nt}}=\frac{1}{2^n}(\frac{n}{p-1})^{p-1}\omega_n.
\end{equation}
Noticing that $\omega_n$ is the value of the relative volume function on hyperbolic space.  We prove that the right side of (\ref{1.11}) is the upper bound of an AH manifold with bounded Ricci curvature from below and we also get a rigidity result. Besides, thanks to the work in \cite{hijazi2020cheeger}, if the AH manifold satisfies the conditions of (\ref{1.8}), we could give the accurate expression of the capacity at infinity. That is,
\begin{theorem}
  Let $(X,g^+)$ be an be an $n+1-$ dimensional AH manifold of $C^{3,\alpha}$ regularity ($\alpha\in (0,1)).$ If the Ricci curvature satisfies $Ric[g^+]+ng^+\geq 0,$ then for any $p>1$ and $q\in X,$
\begin{equation}\label{1.12}
\limsup\limits_{t\rightarrow+\infty}\frac{cap_p(B_{g^+}(q,t))}{e^{nt}}\leq\frac{1}{2^n}(\frac{n}{p-1})^{p-1}\omega_n
\end{equation}
The equality holds if and only if  $(X,g^+)$ is isometric to a hyperbolic space.
\\
   If in addition, the conformal infinity is of nonnegative Yamabe constant and the scalar curvature satisfies $R[g^+]+n(n+1)=o(x^2)$ where $x$ is a defining function. Then
\begin{equation}\label{1.13}
  \lim\limits_{t\rightarrow+\infty}\frac{cap_p(B_{g^+}(q,t))}{e^{nt}}=\frac{1}{2^n}(\frac{n}{p-1})^{p-1}\mathcal{A}(q).
\end{equation}
   for any $p>1$ and $q\in X.$
\end{theorem}
Theorem 1.7 is exactly an extension of (\ref{1.11}) to an AH manifold.
\par In the end, we give an estimate of the capacity of balls in a general AH manifold without any other curvature condition. i.e.
\begin{theorem}
  Let $(X,g^+)$ be an $n+1-$ dimensional AH manifold of $C^2$ regularity, then for any $q\in X,p>1,$
\begin{equation}\label{1.14}
cap_p(B_{g^+}(q,t))=O(e^{nt}),\ \ t\rightarrow+\infty.
\end{equation}
\end{theorem}

The paper is organized as follows. In section 2, we first review the key steps in \cite{dutta2010rigidity} and \cite{li2017gap} proving lemma 1.1. and give a method to compute the function $\mathcal{A}$ in (\ref{2.4}). Then with the help of triangle inequality, we show that
 $\ln\mathcal{A}$ is Lipschitz continuous.
\par In order to prove theorem 1.3 (2), we introduce the concept of the relative volume function acting on a compact set in section 3. Then by studying the relationship between the two functions, we could obtain the upper and lower estimate of $\delta.$ In the end, we prove theorem 1.5.
\par In section 4 we compute the relative volume functions on some AH manifolds. Firstly, we prove that $\mathcal{A} \equiv \omega_n$ on $n+1-$dimensional hyperbolic space. Afterwards, we do some calculations on hyperbolic manifolds and get the expression of $\mathcal{A}.$ In particular, the relative volume function $\mathcal{A}$ of hyperbolic manifolds  $\mathcal{A}<\omega_n.$  And then a rigidity result is proved when the Ricci curvature is greater than or equal to $-n.$ We also prove that the relative volume function $\mathcal{A}$ is not a constant function on hyperbolic manifolds with small $\lambda>0.$
\par In sections 5 and 6, we prove theorem 1.7 and 1.8. The main techniques are classic and are from  Maz’ya's work.
\section{The Lipschitz property of the relative volume function}
We will first give a short proof of the existence of the relative volume fnction $\mathcal{A}(q),$ where more details can be found in \cite{dutta2010rigidity} and \cite{li2017gap}.
\par Let $p\in X$ and $B(t)=B_{g^+}(p,t)$ be the geodesic sphere in $(X,g^+)$ centered at $p$ and set
 \begin{equation}\label{2.1}
   s_p(\cdot)=d_{g^+}(p,\cdot)
 \end{equation}
 be the distance function from $p.$ We state that the symbol $\nabla$ and $|.|$ are the Levi-Civita connection and the norm with respect to the metric $g^+.$ Let $g=x^2g^+$  be the geodesic compactification on $X_\delta=\partial X\times(0,\delta)$ for some small $\delta.$
 We set $r=-\ln x,$ then $|\nabla r|=|\nabla^gx|_g=1$ on $X_\delta,$ which means that $r$ is a distance function on $(X_\delta,g^+).$ By triangle inequality we obtain that $s_p-r$ is bounded on $X_\delta.$ Moreover, if the curvature satisfies condition \ref{1.3} ,then $|d(r-s_p)|_g$ is also bounded when $r$ is large enough and $s_p$ is smooth. (Lemma 4.1 in \cite{li2017gap} and lemma 3.1 in \cite{dutta2010rigidity}).
\par  Consider the conformal metric

\begin{equation}\label{2.2}
  g^p=e^{-2s_p}g^+=e^{2(r-s_p)}g
\end{equation}
on $X_\delta,$ then $g^p$ is Lipschitz continuous to the boundary. Let $\hat{g}^p=g^p|_{\partial X}$ be the boundary metric of $g^p.$ If we regard $\partial B(t)$ as a Lipschitz graph over $\partial X$ in $(X,g^p),$ we could obtain that
\begin{equation}\label{2.3}
\lim\limits_{t\rightarrow+\infty}Vol(\partial B(t),g^p|_{\partial B(t)})=Vol(\partial X,\hat{g}^p)
\end{equation}
    Then by the definition,
\begin{equation}\label{2.4}
\begin{aligned}
\mathcal{A}(p)&=2^n\cdot\lim\limits_{t\rightarrow+\infty}e^{-nt}Vol(\partial B(t),g^+)\\
&=2^n\cdot\lim\limits_{t\rightarrow+\infty}Vol(\partial B(t),g^p|_{\partial B(t)})
\\&=2^n\cdot Vol(\partial X,\hat{g}^p)
\end{aligned}
\end{equation}
Hence $\mathcal{A}(p)>0$ for any $p\in X.$
\par For any $p,q\in X,$ we use $s_p$ and $s_q$ to denote the distance function from $p$ and $q$ respectively in $(X,g^+).$
 Set
\begin{equation}\label{2.5}
  g^p=e^{-2s_p}g^+\ \ and\ \ g^q=e^{-2s_q}g^+
\end{equation}
 Suppose  $d=d_{g^+}(p,q),$ then $|s_p(\cdot)-s_q(\cdot)|\leq d.$
We could use (\ref{2.4}) to get that
\begin{equation}\label{2.6}
\begin{aligned}
\mathcal{A}(p)&=2^n\cdot Vol(\partial X,\hat{g}^p)=2^n\int_{\partial X}dV_{\hat{g}^p}\\
&=2^n\int_{\partial X}e^{-n(s_p-s_q)}dV_{\hat{g}^q}
\end{aligned}
\end{equation}
Hence
\begin{equation}\label{2.7}
e^{-nd}\mathcal{A}(q)\leq\mathcal{A}(p)\leq e^{nd}\mathcal{A}(q)
\end{equation}
It means that
\begin{equation}\label{2.8}
  \frac{|\ln\mathcal{A}(p)-\ln\mathcal{A}(q)|}{d_{g^+}(p,q)}\leq n.
\end{equation}
Therefore $\ln\mathcal{A}$ is Lipschitz continuous.
\\ \par The proof above uses the property that $s_p-s_q$ has a continuous extension to the boundary because the curvature satisfies the condition $(\ref{1.3}).$ In the following, we provide another method to show that $\ln\mathcal{A}$ is still Lipschitz continuous without $(\ref{1.3})$ as long as $\mathcal{A}(p)$ is well defined. The Lipschitz constant is $n+1$ in this case.
\par Firstly, in $X,$ we have that
\begin{equation}\label{2.9}
  \det g^p=e^{-2(n+1)s_p}\det g^+=e^{-2(n+1)(s_p-s_q)}\det g^q\leq e^{2(n+1)d}\det g^q
\end{equation}
 Let $B(t)=B_{g^+}(p,t)$ and $D(t)=B_{g^+}(q,t),$ then for any $t>0$ big enough,
\begin{equation}\label{2.10}
  B(t)\subseteq D(t+d)\ \ \ \ D(\frac{t}{2})\subseteq B(\frac{t}{2}+d)
\end{equation}
 Hence
\begin{equation}\label{2.11}
  \begin{aligned}
Vol(B(t)\setminus B(\frac{t}{2}+d),g^p)&=\int_{B(t)\setminus B(\frac{t}{2}+d)}dV_{g^p}\leq e^{(n+1)d}\int_{B(t)\setminus B(\frac{t}{2}+d)}dV_{g^q}\\
&=e^{(n+1)d}\cdot Vol(B(t)\setminus B(\frac{t}{2}+d),g^q)
\\ &\leq e^{(n+1)d}\cdot Vol(D(t+d)\setminus D(\frac{t}{2}),g^q).
\end{aligned}
\end{equation}
Which implies
\begin{equation}\label{2.12}
  \int_{\frac{t}{2}+d}^t Vol(\partial B(u),g^p)du\leq e^{(n+1)d} \int_{\frac{t}{2}}^{t+d} Vol(\partial D(u),g^q)du.
\end{equation}
Integral mean value theorem tells us that $\exists\xi_t\in[\frac{t}{2}+d,t]$ and $\exists\eta_t\in[\frac{t}{2},t+d]$
such that
\begin{equation}\label{2.13}
 (\frac{t}{2}-d)\cdot Vol(\partial B(\xi_t),g^p)\leq  e^{(n+1)d}\cdot(\frac{t}{2}+d)\cdot Vol(\partial D(\eta_t),g^q)
\end{equation}
Multiply both sides by $\frac{2}{t}$ and let $t$ tends to infinity, we deduce that
\begin{equation}\label{2.14}
  \mathcal{A}(p)\leq e^{(n+1)d}\cdot\mathcal{A}(q)
\end{equation}
We could also prove that $\mathcal{A}(q)\leq e^{(n+1)d}\cdot\mathcal{A}(p)$ in the same method. Hence  $\ln\mathcal{A}$ is Lipschitz continuous and the Lipschitz constant is $n+1.$

\section{Estimates of $\delta:$ the height of the geodesic defining function}
In order to get the upper and lower bounds of $\delta,$ we need to extend the definition of the function $\mathcal{A}.$
Let $E\subset X$ be a compact subset, we use $s_E$ to denote the distance function of $E,$ i.e.
\begin{equation}\label{3.1}
  \forall q\in X,\ \ s_E(q)=d_{g^+}(E,q)=\inf\limits_{y\in X}d_{g^+}(y,q).
\end{equation}
Recall the geodesic compactification $g=x^2g^+$ with boundary metric $\hat{g}$ and let $r=-\ln x$ near boundary. Similar to the beginning of section 2, we set the conformal change $g^E=e^{-2s_E}g^+$ on $X$  and could also prove that $|d(s_E-r)|_{g}$ is bounded  for sufficiently large $r$ almost everywhere. Hence $g^E$ is a Lipschitz continuous metric on $\overline{X}.$ Then we can also define the relative volume function of $E,$
\begin{equation}\label{3.2}
\begin{aligned}
  \mathcal{A}(E)&=2^n\cdot\lim\limits_{t\rightarrow+\infty}e^{-nt}Vol(\partial B_{g^+}(E,t),g^+)\\
&=2^n\cdot\lim\limits_{t\rightarrow+\infty}Vol(\partial B_{g^+}(E,t),g^E|_{\partial B_{g^+}(E,t)})\\
&=2^n\cdot Vol(\partial X,\hat{g}^E)
\end{aligned}
\end{equation}
 where $\hat{g}^E=g^E|_{\partial X}.$
\par Here is the relationship between $\mathcal{A}(p)$ and $\mathcal{A}(E)$:
\begin{lemma}
If $E$ is a compact subset of an AH manifold  $(X^{n+1},g^+),$ then
\begin{equation}\label{3.3}
  \forall p\in E,\mathcal{A}(p)\leq \mathcal{A}(E)\leq e^{n\cdot diam(E,g^+)}\mathcal{A}(p).
\end{equation}
Here $diam(E,g^+)=\sup\limits_{p,q\in E}d_{g^+}(p,q)$ denotes the diameter of $E.$
\end{lemma}
\begin{proof}
  By the definition of distance function, $\forall q\in X,s_p(q)\geq s_E(q).$ As $s_p-s_E$ is  Lipschitz continuous to the boundary,
  $(s_p-s_E)|_{\partial X}\geq 0.$ Noticing that
\begin{equation}\label{3.4}
  \hat{g}^p=e^{-2s_p}g^+=e^{-2(s_p-s_E)}\hat{g}^E,
\end{equation}
we obtain
\begin{equation}\label{3.5}
\begin{aligned}
  \mathcal{A}(p)&=2^n\int_{\partial X}dV_{\hat{g}^p}=2^n\int_{\partial X}e^{-n(s_p-s_E)}dV_{\hat{g}^E}\\
&\leq2^n\int_{\partial X}dV_{\hat{g}^E}=\mathcal{A}(E).
\end{aligned}
\end{equation}
The second inequality could be proved in the same way with the triangle inequality
\begin{equation}\label{3.6}
  s_p(q)\leq s_E(q)+diam(E,g^+),\ \ \ \forall q\in X.
\end{equation}
\end{proof}

Suppose that $g=x^2g^+$ is the $C^2$ geodesic compactification with boundary metric $\hat{g}$ and $|\nabla^{g} x|_g= 1$ on $X_\delta=\partial X\times[0,\delta),$ we need to explain that $\partial X\times\{\delta\}$ may not be homeomorphism to $\partial X$ as it may degenerate into a lower dimension manifold. However we always have that $\partial X\times\{\delta'\}$ is homeomorphism to $\partial X$ for all $\delta'<\delta.$ We are now going to prove that the estimate (\ref{1.6}) holds for all $\delta'<\delta.$
\par Let $E'=X\setminus X_{\delta'}$ be a compact subset. Hence
\begin{equation}\label{3.7}
  \forall p\in X\setminus X_{\delta}\subseteq E',\ \ \ \ \ \mathcal{A}(p)\leq \mathcal{A}(E')\leq e^{n\cdot diam(E',g^+)}\mathcal{A}(p)
\end{equation}

We define function $r:$ $r(\cdot)=\ln\frac{\delta'}{x(\cdot)}$ on $X_{\delta'}$ and $r=0$ on $E'.$
Then $r$ is continuous on $X$ and smooth on $X_{\delta'}.$
Since $E'$ is compact, for any $q\in X_{\delta'},$ there exists a point $q'\in \partial E'=\partial X\times\{\delta'\}$ such that
\begin{equation}\label{3.8}
  d_{g^+}(q,E')=d_{g^+}(q,q')
\end{equation}
For any smooth curve $\sigma:[0,l]\rightarrow X_{\delta'}$ with  $\sigma(0)=q$ and $\sigma(l)=q',$
\begin{equation}\label{3.9}
  \begin{aligned}
Length(\sigma,g^+)&=\int_0^l|\dot{\sigma}(t)|dt=\int_0^l|\nabla r|\cdot|\dot{\sigma}(t)|dt
\\ &\geq\int_0^l|g^+(\nabla r,\dot{\sigma})|dt=\int_0^l|(r\circ\sigma)'(t)|dt
\\ &\geq|\int_0^l (r\circ\sigma)'(t)dt|=|r(q')-r(q)|
  \end{aligned}
\end{equation}
(One can see Chapter 5 in \cite{petersen2006riemannian} for more details.) If we choose $\sigma$ as an integral curve for $\nabla r,$ then $Length(\sigma,g^+)=|r(q')-r(q)|.$
As a consequence,
\begin{equation}\label{3.10}
  d_{g^+}(q,q')=|r(q')-r(q)|=r(q)
\end{equation}
Then $d_{g^+}(q,E')=r(q),$ which means
\begin{equation}\label{3.11}
  r(\cdot)=s_{E'}(\cdot) \ \ on\ \ X.
\end{equation}
\par With the preparations above, we could compute $\mathcal{A}(E').$ We notice that near  the boundary,
\begin{equation}\label{3.12}
  g^{E'}=e^{-2s_{E'}}g^+=e^{-2r}g^+=\frac{x^2}{(\delta')^2}g^+=\frac{1}{(\delta')^2}g
\end{equation}
Therefore, $\hat{g}^{E'}=\frac{1}{(\delta')^2}\hat{g}$ and
\begin{equation}\label{3.13}
  \mathcal{A}(E')=2^n\cdot Vol(\partial X,\hat{g}^E)=(\frac{2}{\delta'})^n\cdot Vol(\partial X,\hat{g})
\end{equation}
Let $\delta'\rightarrow\delta$ and combine it with (\ref{3.7}), we get the estimate of $\delta$ in (\ref{1.6}). At last,
(\ref{1.7}) holds because of the rigidity result, i.e. theorem 1.6 (or theorem 4.5.)
\\
\par At the end of this section, we will prove theorem 1.5. Let's fix a boundary point $\hat{p}\in \partial X$ and select the boundary metric $\hat{g}$ satisfying that the height of the geodesic defining function is bigger than 1, i.e. $\delta>1.$ This always can be realized as we can make a scaling of $\hat{g}.$ Now set $E_1=X\setminus X_1$ and make an estimate of $\mathcal{A}(p_0)$ where $p_0=(\hat{p},x_0)$ for some $x_0\in (0,1)$ small enough. The first thing we need to do is to control $s_{p_0}(q)-s_{E_1}(q)$ when $q$ tends to the boundary $\partial X.$ Assume that $q=(\hat{q},x)$ for some $\hat{q}\in \partial X$ and $x\in (0,x_0)$ and let $q_0=(\hat{q},q_0),$ then
\begin{equation}\label{3.14}
\begin{aligned}
  s_{p_0}(q)-s_{E_1}(q)&=d_{g^+}(p_0,q)-d_{g^+}(q,E_1)
  \\&=d_{g^+}(p_0,q)-d_{g^+}(q,q_0)-d_{g^+}(q_0,E_1)
  \\&\leq d_{g^+}(p_0,q_0)+\ln x_0.
  \end{aligned}
\end{equation}
Let $\sigma (t):[0,1]\rightarrow\partial X$ be the segment connecting $\hat{p}$ and $\hat{q}$ on the boundary, i.e.
$\hat{g}(\dot{\sigma}(t),\dot{\sigma}(t))=d^2_{\hat{g}}(\hat{p},\hat{q})$ for any $t\in [0,1].$ Then define
$\gamma(t)=(\sigma(t),x_0):[0,1]\rightarrow X$ to be a line connecting $p_0$ and $q_0$ on $X\times\{x_0\}.$
Recall that $g^+=x^{-2}(dx^2+\hat{g}+O(x^2))$ near boundary, so
\begin{equation}\label{3.15}
\begin{aligned}
   d_{g^+}(p_0,q_0)&\leq \int_0^1\sqrt{g^+(\dot{\gamma}(t),\dot{\gamma}(t))}dt
   \\& \leq \int_0^1x_0^{-1}\sqrt{\hat{g}(\dot{\sigma}(t),\dot{\sigma}(t))(1+Cx_0^2)}dt
   \\&\leq \frac{d_{\hat{g}}(\hat{p},\hat{q})}{x_0}+C_1x_0.
  \end{aligned}
\end{equation}
Here $C_1$ is a constant depending only on $n$ and $Ric[\hat{g}].$ (\ref{3.14}) and (\ref{3.15}) imply that
 \begin{equation}\label{3.16}
 \begin{aligned}
 \mathcal{A}(p_0)&=2^n\int_{\partial X}dV_{\hat{g}^{p_0}}=2^n\int_{\partial X}e^{-n(s_{p_0}-s_{E_1})|_{\partial X}}dV_{\hat{g}^{E_1}}
 \\ &\geq (\frac{2}{x_0})^n\cdot e^{-nC_1x_0}\int_{\partial X}e^{-\frac{nd_{\hat{g}}(\hat{p},\cdot)}{x_0}}dV_{\hat{g}}
 \\ &\geq (\frac{2}{x_0})^n\cdot e^{-nC_1x_0}\int_{B_{\hat{g}}(\hat{p},R)}e^{-\frac{nd_{\hat{g}}(\hat{p},\cdot)}{x_0}}dV_{\hat{g}}
 \\ &=(\frac{2}{x_0})^n\cdot e^{-nC_1x_0}\int_0^R e^{-\frac{nr}{x_0}}\cdot Vol(\partial B_{\hat{g}}(\hat{p},r)) dr
\end{aligned}
\end{equation}
We choose a fixed small $R>0$ depending on $\hat{p}$ such that 
$$Vol(\partial B_{\hat{g}}(\hat{p},r))\geq \omega_{n-1} r^{n-1}-C_2r^{n+1}$$ for any $r\leq R$ and $C_2$ depends on $\hat{g}.$
Then
 \begin{equation}\label{3.17}
 \begin{aligned}
 \mathcal{A}(p_0)&\geq (\frac{2}{x_0})^n\cdot e^{-nC_1x_0}\int_0^R e^{-\frac{nr}{x_0}}\cdot (\omega_{n-1} r^{n-1}-C_2r^{n+1}) dr
 \\ &=(\frac{2}{x_0})^n\cdot e^{-nC_1x_0} \cdot (\frac{x_0}{n})^n\int_0^{\frac{nR}{x_0}}e^{-t}(\omega_{n-1}t^{n-1}-C_2(\frac{x_0}{n})^2t^{n+1})dt
\end{aligned}
\end{equation}
Let $x_0\rightarrow 0,$ we finally obtain that
 \begin{equation}\label{3.18}
\liminf\limits_{p_0\rightarrow\hat{p}} \mathcal{A}(p_0)\geq(\frac{2}{n})^n\int_0^{+\infty}e^{-t}\omega_{n-1}t^{n-1}dt=(\frac{2}{n})^n(n-1)!\omega_{n-1}
\end{equation}
Then we finish the proof of theorem 1.5 since $\hat{p}$ is arbitrarily selected.
\section{Examples and a rigidity result}
\begin{example}
The hyperbolic space:
\begin{equation}\label{4.1}
(\mathbb{H}^{n+1},g_\mathbb{H}=dt^2+\sinh^2tg_{\mathbb{S}^n})
\end{equation}
 where $g_{\mathbb{S}^n}$ is the standard metric on the $n-$sphere. Let $o$ be the centre point, then the distance function from o, $s_o=t.$
 So
\begin{equation}\label{4.2}
\mathcal{A}(o)=2^n\cdot\lim\limits_{t\rightarrow+\infty}e^{-nt}\omega_n\sinh^nt=\omega_n.
\end{equation}
Here $\omega_n$ denotes the volume of the standard unit sphere $\mathbb{S}^n.$
Since $\mathcal{A}$ is an invariant under the isometry transformation and ${H}^{n+1}$ is a homogeneous space, we derive that
$\mathcal{A}$ is a constant function on $\mathbb{H}^{n+1}$ and $\mathcal{A}\equiv \omega_n.$
\end{example}

\begin{example}\label{ex4.2}
Hyperbolic manifold:
\begin{equation}\label{4.3}
  (\mathbb{R}^n\times \mathbb{S}^1(\lambda),g^+=dt^2+\sinh^2tg_{\mathbb{S}^{n-1}}+\cosh^2tg_{\mathbb{S}^1(\lambda)}).
\end{equation}
Here $n\geq 2,$ and we use the polar coordinate system $\mathbb{R}^n=[0,+\infty)\times\mathbb{S}^{n-1}.$
The conformal infinity of $g^+$ is
$$(\mathbb{S}^{n-1}\times\mathbb{S}^1,[g_{\mathbb{S}^{n-1}}+g_{\mathbb{S}^1(\lambda)}]).$$
Then we have
\begin{equation}\label{4.4}
  \mathcal{A}(t_0,\cdot,\cdot)=\int_{\mathbb{S}^{n-1}\times\mathbb{S}^1(\lambda)}e^{nb_{\gamma_+}(t_0,w,\theta)}d\Theta_{n-1}d\Theta_\lambda.
\end{equation}
Here $d\Theta_{n-1}$ and $d\Theta_\lambda$  are the standard metric on the $n-1$ sphere $\mathbb{S}^{n-1}$ and $\mathbb{S}^1(\lambda).$
$\gamma_+$ is a geodesic ray starting at $(0_n,\theta)$ and $b_{\gamma_+}$ is the buseman function with respect to $\gamma_+.$
\begin{proof}
Let $E=\{0_n\}\times\mathbb{S}^1(\lambda)$ be a compact set where $0_n$ is the centre of $\mathbb{R}^n.$ We know that $t$ is the distance function from $E,$ i.e. $t=s_E.$ Let $g^E=e^{-2t}g^+,$ then
\begin{equation}\label{4.5}
  \hat{g}^E=g^E|_{\partial (\mathbb{R}^n\times \mathbb{S}^1(\lambda))}=\frac{1}{4}(g_{\mathbb{S}^{n-1}}+g_{\mathbb{S}^1(\lambda)}).
\end{equation}

Now we are going to calculate $\mathcal{A}(p_0)$ for some $p_0=(t_0,w_0,\theta_0)\in \mathbb{R}^n\times \mathbb{S}^1(\lambda).$
Firstly we need to study the distance function from $p_0.$ Let $y=(t,w,\theta)\in \mathbb{R}^n\times \mathbb{S}^1(\lambda)$
where $t\geq t_0.$ If we use $d$ to denote the distance function $d_{g^+}$,  then
\begin{equation}\label{4.6}
  d(p_0,y)=d((t_0,w,\theta),(t,w_0,\theta_0))
\end{equation}
We still use $s_{p_0}(\cdot)$ and $s_E(\cdot)$ to denote the distance function from $p_0$ and $E.$ Hence
\begin{equation}\label{4.7}
\lim\limits_{t\rightarrow+\infty}[s_E(y)-s_{p_0}(y)]=\lim\limits_{t\rightarrow\infty}
[t-d((t_0,w,\theta),(t,w_0,\theta_0))]=b_{\gamma_+}(t_0,w,\theta)
\end{equation}
where $\gamma_+(t)=(t,w_0,\theta_0),t\geq 0$ is a geodesic ray.
Then
\begin{equation}\label{4.8}
\begin{aligned}
  \mathcal{A}(p_0)&=2^n\int_{\partial X}dV_{\hat{g}^{p_0}}=2^n\int_{\partial X}e^{-n(s_{p_0}-s_E)|_{\partial X}}dV_{\hat{g}^E}\\
  &=\int_{\mathbb{S}^{n-1}\times\mathbb{S}^1(\lambda)}e^{nb_{\gamma_+}(t_0,w,\theta)}d\Theta_{n-1}d\Theta_\lambda.
  \end{aligned}
\end{equation}
By an orthogonal transformation, we find that $\mathcal{A}(p_0)$ is not depending on $(w_0,\theta_0),$ and hence it is a function of $t_0.$
\end{proof}
\begin{lemma}
If $t_0=0,$ then
\begin{equation}\label{4.9}
  \mathcal{A}(p_0)=2\omega_{n-1}\int_0^{\lambda\pi}\frac{1}{\cosh^n y}dy.
\end{equation}
\end{lemma}

\begin{proof}
Let $p_0=(0_0,\theta_0)$ and $y=(t,w,\theta), t\geq 0.$ Choose $q=(0_n,\theta).$ Then the geodesic line $\overline{p_0q}$ is  perpendicular to $\overline{qy}$ at $q.$ We also know that
\begin{equation}\label{4.10}
d_{g^+}(y,q)=t,\ \ \ \
d_{g^+}(p_0,q)=d_{\mathbb{S}^1(\lambda)}(\theta,\theta_0)=\lambda|\theta-\theta_0|.
\end{equation}
Recall Pythagorean theorem in hyperbolic manifold, i.e.
\begin{equation}\label{4.11}
\cosh t\cdot\cosh (\lambda|\theta-\theta_0|)=\cosh (d_{g^+}(p_0,y)).
\end{equation}
Hence
\begin{equation}\label{4.12}
\begin{aligned}
  b_{\gamma_+}(0_n,\theta)&=\lim\limits_{t\rightarrow\infty}[s_E(y)-s_{p_0}(y)]
  \\&=\lim\limits_{t\rightarrow\infty}[t-\cosh^{-1}(\cosh t\cdot\cosh(\lambda|\theta-\theta_0|))]  \\
  &= -\ln\cosh (\lambda|\theta-\theta_0|).
\end{aligned}
\end{equation}
Then
\begin{equation}\label{4.13}
  \mathcal{A}(p_0)=\int_{\mathbb{S}^{n-1}\times\mathbb{S}^1(\lambda)}\frac{1}{\cosh^n (\lambda|\theta-\theta_0|)}d\Theta_{n-1}d\Theta_\lambda.
\end{equation}
Then (\ref{4.9}) is achieved by the variable substitution of multiple integrals.

\end{proof}
Notice that (\ref{4.9}) can be obtained by partial integration and recursion. Furthermore, we have the following estimate
\begin{equation}\label{4.14}
  0<\mathcal{A}(0_n,\theta_0)<2\omega_{n-1}\int_0^{+\infty}\frac{1}{\cosh^n y}dy=\omega_n.
\end{equation}
We observe that $\mathcal{A}((0_n,\theta_0))$ can take all values in $(0,\omega_n)$ for $\lambda\in(0,+\infty).$
\begin{lemma}
If $\lambda$ is small enough, then $\mathcal{A}$ is not a constant function.
\end{lemma}

\begin{proof}
Theorem 1.5 implies that the function $\mathcal{A}$ is uniformly bounded from below at infinity and the lower bound only depends on the dimension. On the other hand, (\ref{4.9}) tells us that
$\mathcal{A}(0_n,\cdot)$ tends to 0 as $\lambda$ goes to 0. Hence $\mathcal{A}$ is not constant as long as $\lambda$ is sufficiently small.
\end{proof}

\end{example}
From (\ref{4.14})we deduce that $\mathcal{A}(0_n,\theta_0) \nearrow\omega_n$ as $\lambda$ tends to infinity when
the hyperbolic manifold $(\mathbb{R}^n\times \mathbb{S}^1(\lambda),g^+)$ is convergent to the hyperbolic space in pointed Cheeger-Gromov topology. In fact, we have the following rigidity result:
\begin{theorem}[theorem 1.6]
  Let $(X,g^+)$ be an $n+1-$ dimensional AH metric and the Ricci curvature satisfies $Ric[g^+]\geq -ng^+$, then $\mathcal{A}\leq \omega_n$ and the equality holds at one point if and only if $X $is isometric to a hyperbolic space.
\end{theorem}
\begin{proof}
    Noticing that when the Ricci curvature is bounded from below, we can use the Bishop-Gromov volume  comparison theorem, i.e.
  $\frac{Vol(\partial B_{g^+}(p,t))}{Vol(\partial B_\mathbb{H}(0,t))}$ is monotonically decreasing on $t$ and
\begin{equation}\label{4.20}
  \frac{Vol(\partial B_{g^+}(p,t))}{Vol(\partial B_\mathbb{H}(o,t))}\leq 1, \ \ \forall t>0
\end{equation}
 The monotonicity guarantees the existence of limit when $t\rightarrow+\infty$ and we derive that $\mathcal{A}(p)\leq \omega_n.$
  \\ If $\mathcal{A}(p)=\omega_n$ for some $p\in X,$ then we can use the property of monotonically decreasing to get
\begin{equation}\label{4.21}
1=\lim\limits_{t\rightarrow +\infty}\frac{Vol(\partial B_{g^+}(p,t))}{Vol(\partial B_\mathbb{H}(o,t))}\leq\lim\limits_{t\rightarrow 0^+}\frac{Vol(\partial B_{g^+}(p,t))}{Vol(\partial B_\mathbb{H}(o,t))}=1
\end{equation}
  Hence $\frac{Vol(\partial B_{g^+}(p,t))}{Vol(\partial B_\mathbb{H}(o,t))}=1$ for any $t>0,$ as a consequence, $X$ is isometric to a hyperbolic space.
\end{proof}

\section{The capacity on AH manifolds}
The concept of capacity plays an important role in researching the isoperimetric inequalities and Sobolev inequalities in  Euclidean space. Many results are extended to some special manifold. Let's start with the notion on manifold.
\begin{definition}
Let $(M,g)$ be a Riemannian manifold. Suppose that $F\subseteq \Omega\subseteq M$ where $F$ is compact and $\Omega$ is open in $M.$ For each $p\geq 1,$ the p-Capacity is defined by
\begin{equation}\label{5.1}
  cap_p(F,\Omega)=\inf\limits_{u}\int_M|\nabla u|^p dV_g
\end{equation}
where the infimum is taken over all Lipschitz functions $u$ satisfying $u = 1$ on
$F$ and $u$ has compact support in $\Omega.$  If $\Omega=M,$ we write $cap_p(F,M)$ as $cap_p(F).$
\end{definition}
We study the capacity of balls in manifold, the formulas of which in simply connected space forms are already known. One can see section
2.2.4 in \cite{maz2013sobolev} for the Euclidean case.  We can directly get the capacity of balls in hyperbolic space $\mathbb{H}^{n+1},$
that is \cite{grigor1999isoperimetric}:
\begin{equation}\label{5.2}
  cap_p(B_r,B_R)=\omega_n[\int_r^R (\sinh t)^\frac{n}{1-p} dt]^{1-p}
\end{equation}
for $p>1.$ Let $R$ tends to infinity, we get that
\begin{equation}\label{5.3}
  cap_p(B_t)=\frac{1}{2^n}(\frac{n}{p-1})^{p-1}\omega_n e^{nt}+o(e^{nt}),\ \ t\rightarrow+\infty.
\end{equation}

\par Let's first recall some notions in the previous section. Suppose $(X,g^+)$ is an AH manifold and $q\in X.$  We denote the geodesic sphere by $B(t)=B_{g^+}(q,t)$ and claim that the symbol $\nabla ,|.|$ are with respect to $g^+.$
We also use $A(s)=Vol(\partial B(s),g^+)$ to denote the n-dimensional volume of the geodesic sphere.
\par Now we use the classical methods from \cite{polya1951isoperimetric} to get the upper bound of capacity via flux. Let $t$ be a fixed big number and define a function
\begin{equation}\label{5.4}
  f:[t,+\infty)\rightarrow [0,1], \ f(t)=1,f(+\infty)=\lim\limits_{s\rightarrow+\infty}f(s)=0
\end{equation}
Define $u=1$ on $B(t)$ and $u=f(s_q)$ on $X\setminus B(t)$ where $s_q=d_{g^+}(q,\cdot),$ then
\begin{equation}\label{5.5}
  \int_X|\nabla u|^pdV_{g^+}=\int_{X\setminus B(t)}[f'(s)]^p|\nabla s|^pdV_{g^+}=\int_t^{+\infty}[f'(s)]^p A(s)ds
\end{equation}
On the other hand,
\begin{equation}\label{5.6}
\begin{aligned}
1&=-\int_t^{+\infty}f'(s)ds=\int_t^{+\infty}-f'(s)[A(s)]^{\frac{1}{p}}[A(s)]^{-\frac{1}{p}}ds\\
&\leq(\int_t^{+\infty}[f'(s)]^pA(s)ds)^{\frac{1}{p}}(\int_t^{+\infty}[A(s)]^{\frac{1}{1-p}}ds)^{1-\frac{1}{p}}
\end{aligned}
\end{equation}
Thus when
\begin{equation}\label{5.7}
  f(s)=-(\int_t^{+\infty}[A(\tau)]^{\frac{1}{1-p}}d\tau)^{-1}\cdot\int_t^s([A(\tau)]^{\frac{1}{1-p}}d\tau+1
\end{equation}
we have the following equality:
\begin{equation}\label{5.8}
  \int_t^{+\infty}[f'(s)]^pA(s)ds=(\int_t^{+\infty}[A(s)]^{\frac{1}{1-p}}ds)^{1-p}
\end{equation}
\par With the above preparations, we could prove theorem 1.7.
\par We first study the capacity of AH manifold in the case the Ricci curvature is bounded from below by $-n,$ i.e. $Ric[g^+]\geq-ng^+.$ By the volume comparison theorem, $\frac{A(s)}{\sinh^ns}$ is monotonically decreasing to $\mathcal{A}(q)$ as $s\rightarrow +\infty.$ Hence  function $\mathcal{A}$ is well-defined.
$$\forall\varepsilon>0,\exists T<0,\forall t>T,\frac{A(t)}{(e^t/2)^n}<\mathcal{A}(q)+\varepsilon.$$
Thus
\begin{equation}\label{5.9}
\begin{aligned}
(\int_t^{+\infty}[A(s)]^{\frac{1}{1-p}}ds)^{1-p}&\leq\int_t^{+\infty}[(\mathcal{A}(q)+\varepsilon)(\frac{e^s}{2})^n]^{\frac{1}{1-p}}ds)^{1-p}
\\ &=(\mathcal{A}(q)+\varepsilon)\frac{1}{2^n}(\int_t^{+\infty} e^{\frac{ns}{1-p}}ds)^{1-p}
\\ &=\frac{1}{2^n}(\frac{n}{p-1})^{p-1}(\mathcal{A}(q)+\varepsilon)e^{nt}
\end{aligned}
\end{equation}
 Let $\varepsilon\rightarrow 0,$ we deduce that

\begin{equation}\label{5.10}
\limsup\limits_{t\rightarrow+\infty}\frac{cap_p(B_{g^+}(q,t))}{e^{nt}}\leq\frac{1}{2^n}(\frac{n}{p-1})^{p-1}\mathcal{A}(q)
\end{equation}
 Hence the upper bound is achieved. Then the first part of theorem 1.7 holds because of (\ref{5.10}) and theorem 1.6.
\\
 \par It is well known that the lower estimates for $p-$capacity could be described via the isoperimetric function. That is,
 an isoperimetric function $I(\tau)=$ the infimum of $Vol(\partial\Omega)$
for all precompact open set $\Omega\subseteq X$ such that $Vol(\Omega)\geq \tau.$
We now recall Maz’ya's isocapacitary inequalitiese of (2.2.8) in \cite{maz2013sobolev},
\begin{equation}\label{5.11}
  cap_p(F,\Omega)\geq (\int_{\mu(F)}^{\mu(\Omega)}[I(\tau)]^{\frac{p}{1-p}}d\tau)^{1-p}
\end{equation}
for $p>1.$  In our paper, $\mu(\Omega)$ is the $n+1$ dimensional Riemannian volume of $\Omega.$
\par Recall the Cheeger isoperimetric constant of $X,$
\begin{equation}\label{5.12}
\mathcal{C}h(X,g^+)=\inf\limits_{\Omega}\frac{Vol(\partial\Omega)}{Vol(\Omega)}
\end{equation}
 where the infimum is taken over all the compact (smooth) domains in $X.$
 If the Ricci curvature and  the scalar curvature of $X$ satisfies (\ref{1.8}) , then it is shown
 in \cite{hijazi2020cheeger} that the Cheeger isoperimetric constant
\begin{equation}\label{5.13}
\mathcal{C}h(X,g^+)= n\Leftrightarrow Y(\partial X,[\hat{g}])\geq 0.
\end{equation}
 Hence under the conditions of theorem 1.7, we derive that the isoperimetric function of $X$ is $I(\tau)=n\tau$ for any $\tau>0.$  Therefore,
\begin{equation}\label{5.14}
cap_p(B_{g^+}(q,t))\geq (\int_{Vol(B_{g^+}(q,t))}^{+\infty}(n\tau)^{\frac{p}{1-p}}d\tau)^{1-p}=\frac{n^p}{(p-1)^{p-1}}Vol(B_{g^+}(q,t))
\end{equation}
Then
\begin{equation}\label{5.15}
\begin{aligned}
\liminf\limits_{t\rightarrow+\infty}\frac{cap_p(B_{g^+}(q,t))}{e^{nt}}&
\geq\liminf\limits_{t\rightarrow+\infty}\frac{n^p}{(p-1)^{p-1}}\frac{Vol(B_{g^+}(q,t))}{e^{nt}}
\\&=\frac{n^p}{(p-1)^{p-1}}\lim\limits_{t\rightarrow+\infty}\frac{\frac{d}{dt}[Vol(B_{g^+}(q,t))]}{ne^{nt}}
\\&=\frac{1}{2^n}(\frac{n}{p-1})^{p-1}\mathcal{A}(q)
\end{aligned}
\end{equation}
 Finally, (\ref{5.10}) and (\ref{5.15}) imply that
\begin{equation}\label{5.16}
\lim\limits_{t\rightarrow+\infty}\frac{cap_p(B_{g^+}(q,t))}{e^{nt}}=\frac{1}{2^n}(\frac{n}{p-1})^{p-1}\mathcal{A}(q)
\end{equation}

 \section{The capacity of  a general AH manifold}
 In this section, we study the capacity of an AH manifold $(X,g^+).$ The curvature conditions (\ref{1.3}) and (\ref{1.8}) may  not hold, thus the relative volume function $\mathcal{A}$ may not exist. However, we can still estimate the capacity of balls in $X$ with the property of asymptotically hyperbolic.
 \par Let $g=x^2g^+$  be the $C^1$ geodesic compactification on $X_\delta=\partial X\times(0,\delta)$ for some small $\delta<1$ and set $r=-\ln x.$ We know that $r$  is a distance function on $(X_\delta,g^+).$  For any $q\in X,$ let $s_q(\cdot)$ be the distance of function from $q.$ We still have the property that $s_q-r$  is bounded on $X_\delta,$  although $|d(s-r)|_g$ may be unbounded. Assume that $|s_q-r|<D$ on $X_\delta.$
 \par Recall that $B(t)=B_{g^+}(q,t).$
For any $t$ large enough, we have that
$$B(t)\subseteq X_{\varepsilon(t)}\subseteq B(t+2D)$$
and here $\varepsilon(t)=e^{-(t+D)}.$ Then
\begin{equation}\label{6.1}
  cap_p(B(t))\leq cap_p(X_{\varepsilon(t)})\leq cap_p(B(t+2D)).
\end{equation}
Hence
\begin{equation}\label{6.2}
  cap_p(B(t))=O(e^{nt})\Leftrightarrow cap_p(X_{\varepsilon(t)})=O(e^{n\varepsilon(t)}),\ t\rightarrow+\infty.
\end{equation}
\par We will prove that $cap_p(X_{\varepsilon})=O(\varepsilon^{-n}),\ \  \varepsilon\rightarrow 0,$ which  would imply that theorem 1.8 holds.
\par We will use the same method as that in section 5 to prove it. For a small $\varepsilon>0,$ define $f:[0,\varepsilon)\rightarrow [0,1], f(0)=0, \ f(\varepsilon)=1$
and consider the function $u(\cdot)=f(x(\cdot))$ on $X_\varepsilon$ and $u(\cdot)=1$ on $X\setminus X_\varepsilon.$ Then for any $p>1,$

\begin{equation}\label{6.3}
\begin{aligned}
\int_X|\nabla^{g^+} u|_{g^+}^pdV_{g^+}&=\int_{X\varepsilon}x^{p-n-1}|\nabla^gu|_g^pdV_g
\\&=\int_{X\varepsilon}x^{p-n-1}|f'(x)|^p\cdot|\nabla^gx|_g^pdV_g
\\&=\int_0^\varepsilon x^{p-n-1}|f'(x)|^p\cdot Vol(\Sigma_x)dx
\end{aligned}
\end{equation}
where $Vol(\Sigma_x)$ denote the $n-$dimensional volume of the level set $\Sigma_x$ of the geodesic defining function.
Now let $T(x)=x^{p-n-1}\cdot Vol(\Sigma_x).$

\begin{equation}\label{6.4}
\begin{aligned}
1&=\int_0^\varepsilon f'(x)dx=\int_0^\varepsilon f'(x)T^{\frac{1}{p}}(x)T^{-\frac{1}{p}}(x)dx\\
&\leq (\int_0^\varepsilon [f'(x)]^pT(x)dx)^{\frac{1}{p}}(\int_0^\varepsilon T^{\frac{1}{1-p}}(x)dx)^{1-\frac{1}{p}}
\end{aligned}
\end{equation}
The equality holds if and only if $$f(x)=(\int_0^\varepsilon T^{\frac{1}{1-p}}(\tau)d\tau)^{-1}\int_0^xT^{\frac{1}{1-p}}(\tau)d\tau$$ for $x\in[0,\varepsilon].$ Then we infer that

\begin{equation}\label{6.5}
\begin{aligned}
cap_p(X_\varepsilon)&\leq \int_X|\nabla^{g^+} u|_{g^+}^pdV_{g^+}=(\int_0^\varepsilon T^{\frac{1}{1-p}}(x)dx)^{1-p}
\\&=(\int_0^\varepsilon  x^{\frac{p-n-1}{1-p}}\cdot [Vol(\Sigma_x)]^{\frac{1}{1-p}}dx)^{1-p}
\end{aligned}
\end{equation}

$g$ is a $C^1$ geodesic compactification, let $\hat{g}$ be its boundary metric, hence by Gauss lemma,
\begin{equation}\label{6.6}
g=dx^2+g_x=dx^2+\hat{g}+O(x)
\end{equation}
holds near boundary. Then
\begin{equation}\label{6.7}
\frac{\det g_x}{\det \hat{g}}=\det(\hat{g}^{-1}g_x)=\det(I_n+O(x))=1+O(x)
\end{equation}

Therefore
\begin{equation}\label{6.8}
Vol(\Sigma_x)=\int_{\Sigma_x}dV_{g_x}=\int_{\partial X}\sqrt{\frac{\det g_x}{\det \hat{g}}}dV_{\hat{g}}=Vol(\partial X,\hat{g})+O(x)
\end{equation}
Back to (\ref{6.5}), we obtain
\begin{equation}
\begin{aligned}
cap_p(X_\varepsilon)&\leq (\int_0^\varepsilon  x^{\frac{p-n-1}{1-p}}\cdot [Vol(\partial X,\hat{g})+O(x))]^{\frac{1}{1-p}}dx)^{1-p}\\
&=Vol(\partial X,\hat{g})(\int_0^\varepsilon  x^{\frac{p-n-1}{1-p}}(1+O(x))dx)^{1-p}\\
&=Vol(\partial X,\hat{g})(\frac{n}{p-1})^{p-1}\varepsilon^{-n}+O(\varepsilon^{-n+1-p})
\end{aligned}
\end{equation}

So $cap_p(X_\varepsilon)=O(\varepsilon^{-n}),$ and it is easy to get the upper bound
\begin{equation}
\limsup\limits_{\varepsilon\rightarrow 0}\varepsilon^n\cdot cap_p(X_\varepsilon)\leq Vol(\partial X,\hat{g})(\frac{n}{p-1})^{p-1}.
\end{equation}

\bibliographystyle{plain}%

\bibliography{bibfile}

\noindent{Xiaoshang Jin}\\
  School of mathematics and statistics, Huazhong University of science and technology, Wuhan, P.R. China. 430074
 \\Email address: jinxs@hust.edu.cn

\end{document}